\documentclass[12pt]{article}

\usepackage{srcltx}
\usepackage{amsmath,amsthm}
\usepackage{amssymb}
\usepackage{enumerate}

\newtheorem{theorem}{Theorem}
\newtheorem{thm}{Theorem}[section]

\newtheorem{lem}[thm]{Lemma}

\numberwithin{equation}{section}

\begin{document}

\title{\bf\Large Diophantine Approximations \hbox{and the Convergence of Certain Series}}

\author{Alexander Begunts, Dmitry Goryashin}

\date{}

\maketitle

\begin{abstract}Consider two series
$$\sum_{n=1}^\infty\frac{\sin^n\pi\theta n}{n^\alpha},\quad\sum_{n=1}^\infty\frac{\cos^n\pi\theta n}{n^\alpha}.$$
We show that number-theoretical properties of $\theta$ have a strong effect on
the convergence when $0<\alpha\leq 1$. The complete investigation
for $\theta\in\mathbb Q$ is given.
For irrational $\theta$ we prove the result which depends on how well $\theta$
can be approximated with rational numbers, i.e. on its irrationality measure.
We obtain that if $\alpha>\frac12$ then both series converge absolutely for almost all real $\theta$.
Finally, we construct such an everywhere dense set of $\theta$ that both series diverge when $\alpha\leq 1$.
\end{abstract}

%

%
%
%

\section*{Introduction}

Various authors considered series which converge or diverge depending on number-theoretical properties of their parameters.
A century ago  G.H.\,Hardy and J.E.\,Little\-wood \cite{hardylittlewood}, {\it inter alia}, investigated the series
$$\sum_{n=1}^\infty\frac{\sin\pi\theta n^2}{n^\alpha},\quad \sum_{n=1}^\infty\frac{\cos\pi\theta n^2}{n^\alpha}.$$
It follows from their results that these series diverge for all irrational $\theta$ when
$\alpha\leq\frac12,$ and converge for all irrational numbers $\theta$ with bounded
partial quotients (for example, quadratic irrationals) when $\alpha>\frac12$.\footnote{In
fact, it follows from estimates of H.\,Weyl sums, obtained later, and celebrated Roth's
theorem, that both series converge for all irrational algebraic $\theta$ when
$\alpha>\frac12$.} They also constructed examples of irrational $\theta$, for which both
series diverge for all $\alpha\leq 1$.

G.\,I.~Arkhipov and K.\,I.~Oskolkov \cite{arkhipov} considered the
series
$$\sum_{n=1}^\infty\frac{\sin\pi\theta n^k}{n}.$$
Their brilliant result states that it converges {\it for all} real $\theta$ if $k$ is odd.
Moreover, they point out such values of $\theta$ for which the series diverges for any even $k$.

In 2011, E.\,Laeng and V.\,Pata \cite{lang} presented a convergence-divergence test for series of
nonnegative terms which fail to be monotonic. As an application of the test the authors
prove the convergence of one series involving the linear function of $\cos n$ raised to the $n$-th power,
using the fact that the irrationality measure of $\pi$ is finite.

In this paper we combine methods of analysis and number theory to investigate two series
$$\sum_{n=1}^\infty\frac{\sin^n\pi\theta n}{n^\alpha},\quad\sum_{n=1}^\infty\frac{\cos^n\pi\theta n}{n^\alpha}.$$
Since they converge absolutely for $\alpha>1$, from now on we keep in view the case $\alpha\leq 1$.
It will be shown that number-theoretical properties of $\theta$ have a strong effect on
the convergence. In case of rational $\theta$ we prove the following statements.

\begin{theorem}\label{cosratio}
The series $$\sum_{n=1}^\infty\frac{\cos^n\pi\theta n}{n^\alpha},$$ where
$\theta=\frac{p}{q}$ with coprime integers $p,q,$ and $0<\alpha\leq 1,$

$1)$ diverges to $+\infty$ if at least one of $p,q$ is even;

$2)$ converges conditionally if both $p,q$ are odd.
\end{theorem}

\begin{theorem}\label{sinratio}
The series $$\sum_{n=1}^\infty\frac{\sin^n\pi\theta n}{n^\alpha},$$ where $\theta=\frac{p}{q}$ with coprime integers $p,q,$
and $0<\alpha\leq 1$,

$1)$ converges absolutely if $q$ is odd;

$2)$ converges conditionally if $q$ is even and not divisible by $4$;

$3)$ diverges to $+\infty$ if $q$ is divisible by $4$.
\end{theorem}

Next, we proceed to the case of irrational $\theta$.
Recall the following definition from the theory of
Diophantine approximations: the irrationality measure (or irrationality exponent) of real number $\theta$
is the infimum $\mu(\theta)$ of the set of positive real numbers $\mu$ for which
$$0<\left|\theta-\frac{p}{q}\right|<\frac{1}{q^{\mu}}$$
holds for (at most) finitely many solutions $p/q$ with $p$ and $q$ integers.
If no such $\mu$ exists we set $\mu(\theta)=+\infty$ and say that $\theta$ is a Liouville number.

It follows from this definition that if $\theta$ is a number with finite irrationality measure $\mu$, then for all $\delta>0$ there exists a constant $c=c(\theta,\delta)>0$ such that the inequality
$$\left|\theta-\frac{p}{q}\right|>\frac{c}{q^{\mu+\delta}}$$
is true for all integers $p$ and $q\geq 1$.

Note that $\mu(\theta)=1$ if and only if $\theta$ is rational, and famous Roth's theorem \cite[Chapter VI]{Cassels}
states that for irrational algebraic numbers it holds $\mu(\theta)=2$.
For transcendental $\theta$ we have $\mu(\theta)\geq 2$.

\begin{theorem}\label{sincosirr}
Let $\theta$ be an irrational number with finite irrationality measure $\mu=\mu(\theta)$.
If $\alpha>1-\frac{1}{2\mu-2}$, then both series
$$\sum_{n=1}^\infty\frac{\sin^n\pi\theta n}{n^\alpha},\quad \sum_{n=1}^\infty\frac{\cos^n\pi\theta n}{n^\alpha}$$
converge absolutely.
\end{theorem}

As a corollary we immediately obtain that both series converge absolutely for any irrational algebraic number $\theta$
when $\alpha>\frac12$. For instance, using the well known result $\mu(e)=2$ and the recently obtained estimate $\mu(1/\pi)=\mu(\pi)\leq 7.6063\ldots$ from \cite{salikhov} we see that both series
$$
\sum_{n=1}^\infty\frac{\sin^n(\pi en)}{n^\alpha},\quad \sum_{n=1}^\infty\frac{\cos^n(\pi en)}{n^\alpha}
$$
converge absolutely for $\alpha>\frac12$, and both series
$$\sum_{n=1}^\infty\frac{\sin^nn}{n^\alpha},\quad \sum_{n=1}^\infty\frac{\cos^nn}{n^\alpha}$$
converge absolutely for $\alpha>0.9243\ldots$.

Moreover, the key result of continued fractions metric theory (see \cite[\S\,14, Theorem 32]{khinchin})
states that $\mu(\theta)=2$ holds for almost all real $\theta$. Therefore from Theorem 3 one can deduce
the upcoming assertion.

\begin{theorem}\label{metric}
Let $\alpha>\frac12$. Then both series
$$\sum_{n=1}^\infty\frac{\sin^n \pi\theta n}{n^\alpha},\quad \sum_{n=1}^\infty\frac{\cos^n\pi\theta n}{n^\alpha}$$
converge absolutely for almost all real $\theta$.
\end{theorem}

Yet, we submit another (and more simple) proof which is independent of Theorem 3 and involves methods of real analysis.

In addition to the investigation of the case of irrational $\theta$, we construct the set
of transcendental numbers for which both series converge only for $\alpha>1$.

\begin{theorem}\label{divergence}
There exists such an everywhere dense in $\mathbb R$ set of irrational numbers $E$,
that for any $\xi\in E$ both series
$$\sum_{n=1}^\infty\frac{\sin^n\pi\xi n}{n^{\alpha}},\quad \sum_{n=1}^\infty\frac{\cos^n\pi\xi n}{n^{\alpha}}$$
diverge for all $\alpha\leq 1$.
\end{theorem}

\section{Proof of Theorem 1}

1) Decompose the series into the sum of $q$ series each
corresponding to values $n\equiv a\pmod q$, $a=1,2,\ldots,q$:
$$\sum_{n=1}^\infty\frac{\cos^n\frac{\pi pn}{q}}{n^\alpha}=\sum_{a=1}^q\sum_{l=0}^\infty \frac{\left(\cos\frac{\pi p(lq+a)}{q}\right)^{lq+a}}{(lq+a)^\alpha}=\sum_{a=1}^q\sum_{l=0}^\infty \frac{\left(\cos(\frac{\pi ap}{q}+\pi pl)\right)^{lq+a}}{(lq+a)^\alpha}.$$

We start with $p$ even, $p=2p_1$. Since $(p,q)=1$, we obtain that $q$ is odd. Among the
numbers $\cos\frac{2\pi ap_1}{q}$, $a=1,2,\ldots,q$, all but one lie in the interval
$(-1;1)$. Indeed, if $a=q$, then $\cos\frac{2\pi ap_1}{q}=\cos 2\pi p_1=1$; for
$1\leq a<q$ we have $\frac{2 ap_1}{q}$ noninteger, hence $\cos\frac{2\pi
ap_1}{q}\neq\pm 1$. Thus in this case the series can be written as
$$\sum_{a=1}^q \sum_{l=0}^\infty \frac{\left(\cos\frac{2\pi
ap_1}{q}\right)^{lq+a}}{(lq+a)^\alpha}=\sum_{a=1}^{q-1}{\textstyle\cos^a\frac{2\pi
ap_1}{q}}\sum_{l=0}^\infty \frac{\left(\cos^q\frac{2\pi
ap_1}{q}\right)^{l}}{(lq+a)^\alpha}+\sum_{l=0}^\infty \frac{1}{q^\alpha(l+1)^\alpha},$$
Since $|\cos^q\frac{2\pi ap_1}{q}|<1$, the first $q-1$ series in the sum converge absolutely,
and the last series diverges for $0<\alpha\leq 1$. Therefore the series under consideration
diverges to $+\infty$.

Next, let $q=2q_1$ and $p$ is odd. The equality $\cos(\frac{\pi ap}{2q_1}+\pi pl)=\pm 1$
is equivalent to $ap\equiv 0\pmod {2q_1}$ which holds only when $a=q=2q_1$. Hence the series
can also be represented as
$$\sum_{a=1}^{2q_1} \sum_{l=0}^\infty \frac{\left(\cos(\frac{\pi ap}{2q_1}+\pi pl)\right)^{2lq_1+a}}{(2lq_1+a)^\alpha}=$$
$$=\sum_{a=1}^{2q_1-1}\sum_{l=0}^\infty \frac{\left(\cos^{2q_1}\frac{\pi ap}{2q_1}\right)^{l}\cos^a(\frac{\pi ap}{2q_1}+\pi pl)}{(2lq_1+a)^\alpha}+\sum_{l=0}^\infty \frac{1}{q^\alpha(l+1)^\alpha}.$$
In the same way as in the previous case we observe that the series diverges to $+\infty$.

2) Let both $p,q$ be odd. Decompose the series into the sum of $2q$ series each
corresponding to values $n\equiv a\pmod {2q}$, $a=1,2,\ldots,2q$:
$$\sum_{n=1}^\infty\frac{\cos^n\frac{\pi pn}{q}}{n^\alpha}=\sum_{a=1}^{2q}\sum_{l=0}^\infty \frac{\left(\cos\frac{\pi p(2lq+a)}{q}\right)^{2lq+a}}{(2lq+a)^\alpha}=\sum_{a=1}^{2q}\sum_{l=0}^\infty \frac{\left(\cos\frac{\pi ap}{q}\right)^{2lq+a}}{(2lq+a)^\alpha}.$$

The equality $\cos\frac{\pi ap}{q}=1$ is equivalent to $ap\equiv 0\pmod{2q}$ which holds only when
$a=2q$, because $p$ and $2q$ are coprime. Next, $\cos\frac{\pi ap}{q}=-1$ means that
$ap\equiv q\pmod{2q}$, which has a unique solution $a=a_0$, $1\leq a_0\leq 2q-1$. Note that $a_0$ is odd.
Hence $|\cos\frac{\pi ap}{q}|<1$ whenever $a\neq a_0$ and $a\neq 2q$.
Rewriting the series as
$$\sum_{a=1}^{2q}\sum_{l=0}^\infty \frac{\left(\cos\frac{\pi ap}{q}\right)^{2lq+a}}{(2lq+a)^\alpha}=$$
$$=\sum_{\substack{a=1\\a\neq a_0}}^{2q-1}{\textstyle \cos^a\frac{\pi ap}{q}}\sum_{l=0}^\infty \frac{\left(\cos^{2q}\frac{\pi ap}{q}\right)^{l}}{(2lq+a)^\alpha}+\sum_{l=0}^\infty \left(\frac{1}{(2ql+2q)^\alpha}-\frac{1}{(2ql+a_0)^\alpha}\right),$$
we conclude that $2q-2$ series in the first sum converge absolutely.
Now
$$\frac{1}{(2ql+2q)^\alpha}-\frac{1}{(2ql+a_0)^\alpha}=\frac{1}{(2ql)^\alpha}
\left(\left(1+\frac{1}{l}\right)^{-\alpha}-\left(1+\frac{a_0}{2ql}\right)^{-\alpha}\right)=$$
$$=\frac{1}{(2ql)^\alpha} \left(1-\frac{\alpha}{l}+O\left(\frac{1}{l}\right)-1+\frac{\alpha a_0}{2ql}+O\left(\frac{1}{l}\right)\right)\sim \frac{\alpha(\frac{a_0}{2q}-1)}{(2q)^\alpha}\cdot\frac{1}{l^{\alpha+1}}$$
for $l\to\infty$. Since $\alpha+1>1,$ the series converges.
It remains to be seen that the series of absolute values
$$\sum_{a=1}^{2q}\sum_{l=0}^\infty \frac{\left|\cos\frac{\pi ap}{q}\right|^{2lq+a}}{(2lq+a)^\alpha}=$$
$$=\sum_{\substack{a=1\\a\neq a_0}}^{2q-1}{\textstyle \cos^a\frac{\pi ap}{q}}\sum_{l=0}^\infty \frac{\left(\cos^{2q}\frac{\pi ap}{q}\right)^{l}}{(2lq+a)^\alpha}+\sum_{l=0}^\infty \left(\frac{1}{(2ql+2q)^\alpha}+\frac{1}{(2ql+a_0)^\alpha}\right),$$
evidently diverges, thus the series under consideration converges conditionally.

\section{Proof of Theorem 2}

By analogy with the previous proof, decompose the series into the sum of $2q$ series:
$$\sum_{n=1}^\infty\frac{\sin^n\frac{\pi pn}{q}}{n^\alpha}=\sum_{a=1}^{2q}\sum_{l=0}^\infty \frac{\left(\sin\frac{\pi p(2lq+a)}{q}\right)^{2lq+a}}{(2lq+a)^\alpha}=\sum_{a=1}^{2q}\sum_{l=0}^\infty \frac{\left(\sin\frac{\pi ap}{q}\right)^{2lq+a}}{(2lq+a)^\alpha}.$$

1) If $q$ is odd, then the equality $\frac{\pi ap}{q}=\frac{\pi}{2}+\pi k$ means that
$2ap=q+2kq$, which is impossible for any integer $k$. Hence in this case $|\sin\frac{\pi
ap}{q}|<1$ for all $a=1,2,\ldots,2q$, thus all the series in the right-hand side converge absolutely,
and so does the series under study.

2) Let $q=2q_1$, where $q_1$ is odd. Then the equality $\frac{\pi
ap}{2q_1}=\frac{\pi}{2}+2\pi k$ holds if and only if $ap\equiv q_1\pmod {4q_1}$. Since
$(p,4q_1)=1$, the congruence has a unique solution $a\equiv a_0\pmod {4q_1}$ where
$1\leq a_0< 4q_1=2q$. In a similar way, $\frac{\pi ap}{2q_1}=-\frac{\pi}{2}+2\pi k$
is equivalent to $ap\equiv -q_1\pmod {4q_1}$, which is true only for $a=2q-a_0$.
Note that $a_0$ is odd. Then for $a\neq a_0$, $a\neq 2q-a_0$ it holds
$|\sin\frac{\pi ap}{q}|<1$, so we can rewrite the series as
$$\sum_{\substack{a=1\\a\neq a_0\\a\neq 2q-a_0}}^{2q}\sum_{l=0}^\infty \frac{\left(\sin\frac{\pi ap}{q}\right)^{2lq+a}}{(2lq+a)^\alpha}+\sum_{l=0}^\infty\left(\frac{1}{(2lq+a_0)^\alpha}-\frac{1}{(2lq+2q-a_0)^\alpha}\right).$$
and see that $2q-2$ series in the first sum converge absolutely. The last series is also convergent as far as
$$\frac{1}{(2ql+a_0)^\alpha}-\frac{1}{(2ql+2q-a_0)^\alpha}=\frac{1}{(2ql)^\alpha}
\left(\left(1+\frac{a_0}{2ql}\right)^{-\alpha}-\right.$$
$$\left.-\left(1+\frac{2q-a_0}{2ql}\right)^{-\alpha}\right)=\frac{1}{(2ql)^\alpha} \left(1-\frac{a_0\alpha}{2ql}+O\left(\frac{1}{l}\right)-1+\frac{\alpha (2q-a_0)}{2ql}+\right.$$
$$\left.+O\left(\frac{1}{l}\right)\right)\sim \frac{\alpha(1-\frac{a_0}{q})}{(2q)^\alpha}\cdot\frac{1}{l^{\alpha+1}}\quad\mbox{ when }\quad l\to\infty.$$
It converges conditionally, since the series of absolute values
$$\sum_{n=1}^\infty\frac{|\sin\frac{\pi pn}{q}|^n}{n^\alpha}=$$
$$=\sum_{\substack{a=1\\a\neq a_0\\a\neq 2q-a_0}}^{2q}\sum_{l=0}^\infty \frac{\left|\sin\frac{\pi ap}{q}\right|^{2lq+a}}{(2lq+a)^\alpha}+\sum_{l=0}^\infty\left(\frac{1}{(2lq+a_0)^\alpha}+\frac{1}{(2lq+2q-a_0)^\alpha}\right)$$
diverges.

3) After all, let $q=4q_1$. Then the equality $\frac{\pi ap}{4q_1}=\frac{\pi}{2}+\pi k$
is equivalent to $ap\equiv 2q_1\pmod {4q_1}$. Since $(p,4q_1)=1$, the congruence has a unique solution
$a\equiv a_0\pmod {4q_1}$ where $1\leq a_0< 4q_1=q$ and $a_0$ is even.
Thus the series under study can be rewritten as
$$\sum_{\substack{a=1\\a\neq a_0\\a\neq q+a_0}}^{2q}\sum_{l=0}^\infty \frac{\left(\sin\frac{\pi ap}{q}\right)^{2lq+a}}{(2lq+a)^\alpha}+\sum_{l=0}^\infty\left(\frac{1}{(2lq+a_0)^\alpha}+\frac{1}{(2lq+q+a_0)^\alpha}\right),$$
and therefore diverges to $+\infty$. This completes the proof.

\section{Proof of Theorem 3}

We assume that $\theta>0$ and prove that the series
$$\sum_{n=1}^\infty\frac{|\cos(\pi n\theta)|^n}{n^\alpha}$$
converges for $\alpha>1-\frac{1}{2\mu-2}$. The sin-series can be considered in the same way.
Since this is the series of absolute values, by the well-known fact its
convergence is independent of the permutation of its terms.
Thus the series can be represented as
$$\sum_{n=1}^\infty\frac{|\cos(\pi n\theta)|^n}{n^\alpha}=\sum_{s=0}^\infty\sum_{\substack{n\in\mathbb{N}\\1-\frac{1}{2^{s}}\leq |\cos(\pi n\theta)|<1-\frac{1}{2^{s+1}}}}\frac{|\cos(\pi n\theta)|^n}{n^\alpha}=\sum_{s=0}^\infty S_{\frac{1}{2^{s}}},$$
where
$$S_\varepsilon=\sum_{\substack{n\in\mathbb{N}\\1-\varepsilon\leq |\cos(\pi n\theta)|< 1-\varepsilon/2}}\frac{|\cos(\pi n\theta)|^n}{n^\alpha}.$$

We need an upper estimate for $S_\varepsilon$, where $\varepsilon$ is fixed and small
enough (greater values of $\varepsilon$ have an effect on at most finitely many terms of the
series $\sum_{s=0}^\infty S_{\frac{1}{2^{s}}}$). For that purpose we prove the following statement.

\begin{lem}\label{gaps}
Let $\theta$ be a positive irrational number with finite irrationality measure $\mu=\mu(\theta)$, $\nu=\frac{1}{2\mu-2}$, $0<\varepsilon<\min(1;1-\cos\frac{\pi\theta}{2})$, and let $\{n_k\}_{k=1}^\infty$ be the increasing sequence of positive integers for which
$$1-\varepsilon\leq |\cos(\pi n\theta)|< 1-\varepsilon/2.$$
Then for all $\delta\in(0;\nu)$ there exists such a number $A=A(\theta,\delta)>0$ (independent of $\varepsilon$) that for all positive integers $k$ the following inequality
holds: $$n_k\geq Ak\varepsilon^{-\nu+\delta}.$$
\end{lem}

\begin{proof}
Let $m=n_k$ and $n=n_{k+1}$ be two neighbor terms of $\{n_k\}_{k=1}^\infty$.
Since $|\cos(\pi n\theta)|\geq 1-\varepsilon$ and $|\cos(\pi m\theta)|\geq 1-\varepsilon$, we obtain
$$-\arccos(1-\varepsilon)+\pi l_n\leq \pi n\theta\leq \arccos(1-\varepsilon)+\pi l_n,$$
$$-\arccos(1-\varepsilon)+\pi l_m\leq \pi m\theta\leq \arccos(1-\varepsilon)+\pi l_m$$
for some $l_n,l_m\in\mathbb{N}$, so
$$|n\theta-l_n|\leq\frac{1}{\pi}\arccos(1-\varepsilon),\qquad |m\theta-l_m|\leq\frac{1}{\pi}\arccos(1-\varepsilon).$$
Then
$$|(n-m)\theta-l|\leq|n\theta-l_n|+|m\theta-l_m|\leq\frac{2}{\pi}\arccos(1-\varepsilon),$$
where $l=l_n-l_m\geqslant0$. Note that if $\varepsilon<\min(1;1-\cos\frac{\pi\theta}{2})$ then
$l_n=l_m$ can hold only when $(n-m)\theta<\frac{2}{\pi}\arccos(1-\varepsilon)<\frac{2}{\pi}\frac{\pi\theta}{2}=\theta$,
i.e. $n-m<1$, a contradiction. Hence, $l=l_n-l_m>0$. Dividing by $l\theta>0$ we get an estimate
$$\left|\frac{n-m}{l}-\frac{1}{\theta}\right|\leq\frac{2}{l\pi\theta}\arccos(1-\varepsilon)<\frac{4\sqrt{\varepsilon}}{l\pi\theta}.$$
Here we also used the inequality $\arccos(1-\varepsilon)<2\sqrt{\varepsilon}$ which holds
when $0<\varepsilon<1$. Indeed, for the function $f(x)=\arccos(1-x)-2\sqrt{x}$ we have
$f'(x)=\frac{1}{\sqrt{x(2-x)}}-\frac{1}{\sqrt{x}}<0$ when $x\in(0;1)$, so $f(x)$ is
decreasing in this interval, thus $f(x)<f(0)=0$.

Alternately, $\frac{1}{\theta}$ has a finite irrationality measure $\mu$, thus for all $\delta'>0$ the inequality
$$\left|\frac{n-m}{l}-\frac{1}{\theta}\right|>\frac{c}{l^{\mu+\delta'}}$$
is true for all integers $l\geq 1$, where $c=c(\theta,\delta')$. Hence we have
$$\frac{c}{l^{\mu+\delta'}}<\frac{4\sqrt{\varepsilon}}{l\pi\theta},\qquad l>\left(\frac{\pi\theta c}{4\sqrt{\varepsilon}}\right)^\frac{1}{\mu-1+\delta'}.$$
Therefore,
$$n_{k+1}-n_k=n-m\geq \frac{l-1}{\theta}\geq
\frac{l}{2\theta}>A\varepsilon^{-\frac{1}{2\mu-2+2\delta'}}=A\varepsilon^{-\nu+\delta},$$ where
$\delta'=\frac{\delta}{2\nu(\nu-\delta)}>0$ and $A=\frac{1}{2\theta}\left(\frac{\pi\theta c}{4}\right)^\frac{1}{\mu-1+\delta'}$ is
independent of $\varepsilon$.

Finally, we note that an analogous argument leads to the same estimate for $n_1$.
Thus the statement of lemma follows by induction.
\end{proof}

Next, we apply Lemma \ref{gaps} to handle $S_\varepsilon$. Since
$n_k\geq Ak\varepsilon^{-\nu}$ we have
$$S_\varepsilon=\sum_{\substack{n\in\mathbb{N}\\1-\varepsilon\leq |\cos(\pi n\theta)|< 1-\varepsilon/2}}\frac{|\cos(\pi n\theta)|^n}{n^\alpha}=\sum_{k=1}^\infty\frac{|\cos(\pi n_k\theta)|^{n_k}}{n_k^\alpha}\ll$$
$$\ll\sum_{k=1}^\infty\frac{(1-\frac{\varepsilon}{2})^{n_k}}{n_k^\alpha}\leq \sum_{k=1}^\infty\frac{(1-\frac{\varepsilon}{2})^{Ak\varepsilon^{-\nu+\delta}}}{(Ak\varepsilon^{-\nu+\delta})^\alpha}= \frac{1}{{A^\alpha \varepsilon^{-(\nu+\delta)\alpha}}}\sum_{k=1}^\infty\frac{\bigl((1-\frac{\varepsilon}{2})^{A\varepsilon^{-\nu+\delta}}\bigr)^k}{k^\alpha}.$$

For $\alpha=1$ and $\alpha<1$ we suggest different approaches.

Let $\alpha=1$. Since
$$\sum_{k=1}^\infty\frac{z^k}{k}=-\ln(1-z),\quad |z|<1,$$
we set $z=(1-\frac{\varepsilon}{2})^{A\varepsilon^{-\nu+\delta}}$ and obtain
$$S_\varepsilon\ll-\frac{1}{{A \varepsilon^{-\nu+\delta}}}\ln\left(1-(1-\tfrac{\varepsilon}{2})^{A\varepsilon^{-\nu+\delta}}\right)=-\frac{1}{{A \varepsilon^{-\nu+\delta}}}\ln\left(1-e^{A\varepsilon^{-\nu+\delta}\ln(1-\frac{\varepsilon}{2})}\right)=$$
$$=-\frac{1}{{A \varepsilon^{-\nu+\delta}}}\ln\left(1-e^{-\frac{A}{2}\varepsilon^{-\nu+\delta+1}(1+o(1))}\right)=-\frac{1}{{A \varepsilon^{-\nu+\delta}}}\ln\left(\tfrac{A}{2}\varepsilon^{-\nu+\delta+1}(1+o(1))\right)\sim$$
$$\sim-\frac{(1-\nu+\delta)\ln\varepsilon}{{A \varepsilon^{-\nu+\delta}}}\quad\mbox{ when }\quad\varepsilon\to+0.$$
Hence the series
$$\sum_{n=1}^\infty\frac{|\cos(\pi n\theta)|^n}{n}=\sum_{s=0}^\infty S_{\frac{1}{2^{s}}}\ll \sum_{s=0}^\infty\frac{(1-\nu+\delta)s\ln 2}{A(2^{\nu-\delta})^s}$$
is convergent because $2^{\nu-\delta}=2^{\frac{1}{2\mu-2}-\delta}>1$ for sufficiently small $\delta>0$.

Now let $\alpha<1$.
We formulate a special case of an asymptotic formula from \cite[\S\,4, paragraph 3]{gelfond} as a lemma.

\begin{lem}
For $\alpha<1$ it holds
$$\sum_{k=1}^\infty\frac{z^k}{k^\alpha}\sim\Gamma(1-\alpha)\left(\ln\frac{1}{z}\right)^{\alpha-1}\quad \mbox{ when }\quad z\to 1-0.$$
\end{lem}

Again put $z=(1-\frac{\varepsilon}{2})^{A\varepsilon^{-\nu+\delta}}$ and for $\alpha<1$ we have
$$S_\varepsilon\ll \frac{1}{{A^\alpha \varepsilon^{(-\nu+\delta)\alpha}}}\Gamma(1-\alpha)\left(-A\varepsilon^{-\nu+\delta}\ln(1-\tfrac{\varepsilon}{2})\right)^{\alpha-1}\sim$$
$$\sim\frac{\Gamma(1-\alpha)}{2^{\alpha-1}A}\varepsilon^{(1-\nu+\delta)(\alpha-1)-(-\nu+\delta)\alpha}= \frac{\Gamma(1-\alpha)}{2^{\alpha-1}A}\varepsilon^{\alpha+\nu-\delta-1}\quad\mbox{ when }\quad\varepsilon\to+0.$$
Thus the series
$$\sum_{n=1}^\infty\frac{|\cos(\pi n\theta)|^n}{n^{\alpha}}=\sum_{s=1}^\infty S_{\frac{1}{2^{s}}}\ll \frac{\Gamma(1-\alpha)}{2^{\alpha-1}A}\sum_{s=1}^\infty\frac{1}{(2^{\alpha+\nu-\delta-1})^s}$$
converges if $2^{\alpha+\nu-\delta-1}>1$, i.e. if $\alpha>1-\nu+\delta=1-\frac{1}{2\mu-2}+\delta$ for sufficiently small $\delta>0$.
Proof of Theorem 3 is completed.

\section{Analytic proof of Theorem 4}

For the reason of periodicity it is enough to prove the convergence for almost all real $\theta$ from $[0;1]$ when $\alpha>\frac12$.
We consider the sin-series; one can easily handle the cos-series in a similar way. Since
$$\int\limits_0^1 \frac{|\sin(\pi n\theta)|^n}{n^\alpha}\,d\theta=\frac{1}{\pi n^{\alpha+1}}\int\limits_0^{\pi n}|\sin x|^n\,dx=\frac{2n}{\pi n^{\alpha+1}}\int\limits_0^{\pi/2}\sin^n x\,dx=\frac{2I_n}{\pi n^\alpha},$$
where
$$I_n=\int\limits_0^{\pi/2}\sin^n x\,dx\asymp \frac{1}{\sqrt{n}},$$
it holds
$$\int\limits_0^1 \frac{|\sin(\pi n\theta)|^n}{n^\alpha}\,d\theta\asymp\frac{1}{n^{\alpha+\frac12}},$$
hence the series with integral terms
$$\sum_{n=1}^\infty\int\limits_0^1 \frac{|\sin(\pi n\theta)|^n}{n^\alpha}\,d\theta$$
converges for $\alpha+\frac12>1$, i.e. for $\alpha>\frac12$. By the corollary of B.\,Levi's monotone convergence theorem \cite[sec. 30.1]{kolmogorov} it follows that the series
$$\sum_{n=1}^\infty\frac{|\sin(\pi n\theta)|^n}{n^\alpha}$$
converges for $\alpha>\frac12$ almost everywhere on $[0;1]$ and therefore on $\mathbb{R}$.

\section{Divergence of the series on an everywhere dense set}

In this section we reveal that every interval contains irrational numbers $\theta$ with infinite
irrationality measure for which both series diverge for all $\alpha\leq 1$. This shows
that the result analogous to Theorem \ref{sincosirr} doesn't hold for such numbers.
First, we need the following lemma on the rate of divergence of the series under
consideration for some rational $\theta$.

\begin{lem}\label{divrate}
Let $\theta=\frac{p}{q}$ with coprime numbers $p,q$, where $q$ is divisible by $4$. Then there exists such a positive integer $A_q$ that the following estimates hold uniformely for all positive integers $L$ and $\alpha\leq 1$:
$$
\left|\sum_{n=1}^{2qL}\frac{\sin^n\frac{\pi pn}{q}}{n^\alpha}\right|\geq\frac{1}{q}\ln L-A_q,
\quad
\left|\sum_{n=1}^{2qL}\frac{\cos^n\frac{\pi pn}{q}}{n^\alpha}\right|\geq\frac{1}{q}\ln L-A_q,
$$
\end{lem}

\begin{proof}
In the proof of Theorem \ref{sinratio} we established the following representation of the $\sin$-series partial sum:
$$\sum_{n=1}^{N_1}\frac{\sin^n\frac{\pi pn}{q}}{n^\alpha}=
\sum_{\substack{a=1\\a\neq a_0\\a\neq q+a_0}}^{2q}\sum_{l=0}^{L} \frac{\left(\sin\frac{\pi ap}{q}\right)^{2lq+a}}{(2lq+a)^\alpha}+\sum_{l=0}^{L}\left(\frac{1}{(2lq+a_0)^\alpha}+\frac{1}{(2lq+q+a_0)^\alpha}\right).$$
Since for $a\neq a_0$, $a\neq q+a_0$ all the series
$$\sum_{l=0}^{\infty} \frac{\left(\sin\frac{\pi ap}{q}\right)^{2lq+a}}{(2lq+a)^\alpha}$$
converge absolutely for all $\alpha$, there exists such an integer $A_q$ depending only on $q$ that
$$
\left|\sum_{\substack{a=1\\a\neq a_0\\a\neq q+a_0}}^{2q}\sum_{l=0}^{L} \frac{\left(\sin\frac{\pi ap}{q}\right)^{2lq+a}}{(2lq+a)^\alpha}\right|\leq
\sum_{\substack{a=1\\a\neq a_0\\a\neq q+a_0}}^{2q}\sum_{l=0}^{\infty} \frac{\left|\sin\frac{\pi ap}{q}\right|^{2lq+a}}{(2lq+a)^\alpha}\leq A_q.
$$
Now it can be easily seen that
$$
\sum_{l=0}^\infty\left(\frac{1}{(2lq+a_0)^\alpha}+\frac{1}{(2lq+q+a_0)^\alpha}\right)\geq
\sum_{l=0}^{L}\left(\frac{1}{2lq+2q}+\frac{1}{2lq+q+q}\right)=$$
$$=\frac{1}{q}\sum_{l=0}^{L}\frac{1}{l+1}\geq \frac{1}{q}\sum_{l=1}^{L}\frac{1}{l}\geq\frac{1}{q}\ln L.
$$
Thus
$$
\left|\sum_{n=1}^{2qL}\frac{\sin^n\frac{\pi pn}{q}}{n^\alpha}\right|\geq\frac{1}{q}\ln L-A_q.
$$

The analogous inequality for the $\cos$-series can be proved in the same way using proof of Theorem \ref{cosratio}.
\end{proof}

\begin{proof}[Proof of Theorem 5]
Our goal is to show that every interval $(x_1;x_2)\subset\mathbb R$ contains such an irrational
number $\xi$ that both series
$$\sum_{n=1}^\infty\frac{\sin^n\pi\xi n}{n^{\alpha}},\quad \sum_{n=1}^\infty\frac{\cos^n\pi\xi n}{n^{\alpha}}$$
diverge for all $\alpha\leq 1$.
This construction uses the idea similar to the one implemented in \cite{hardylittlewood}.
We start with choosing an integer $u$ and positive integers $b_1<b_2<\ldots<b_s<\nu_1$
for which both numbers
$$
\xi_0=u+\sum_{i=1}^{s}\frac{1}{2^{b_i}} \quad\mbox{ and }\quad \xi_0+\frac{2}{2^{\nu_1}}
$$
fall in the interval $(x_1;x_2)$; evidently, these numbers exist.
Let $$\xi_1=\xi_0+\frac{1}{2^{\nu_1}},$$
then since $\xi_1=\frac{p}{q},$ where $p$ is odd and $q=2^{\nu_1}$ is divisible by $4$, both series
$$\sum_{n=1}^\infty\frac{\sin^n\pi\xi_1 n}{n^{\alpha}},\quad \sum_{n=1}^\infty\frac{\cos^n\pi\xi_1 n}{n^{\alpha}}$$
diverge to $+\infty$ by virtue of Theorems 1 and 2.
We recurrently define $\nu_2,\nu_3,\ldots$ so that
$$\nu_{k+1}=2^{\nu_{k}+2}(A_{2^{\nu_{k}}}+k+4)+2\nu_{k}+3,$$
where $A_{2^{\nu_{k}}}$ is provided by Lemma \ref{divrate}, and set
$$\xi_k=\xi_0+\sum_{i=1}^{k}\frac{1}{2^{\nu_i}}\quad\mbox{ and }\quad\xi=\xi_0+\sum_{i=1}^{\infty}\frac{1}{2^{\nu_i}},$$
therefore $\xi\in(x_1;x_2)$, which satisfies our needs.

Now we are ready to show that for every positive integer $k$ there exists such a number $N_k$ that
$$
\left|\sum_{n=1}^{N_k}\frac{\sin^n \pi\xi n}{n^{\alpha}}\right|> k \quad\mbox{ and }\quad \left|\sum_{n=1}^{N_k}\frac{\cos^n \pi\xi n}{n^{\alpha}}\right|> k,
$$
which means the desired divergence.
We have
$$
\left|\sum_{n=1}^{N_k}\frac{\sin^n \pi\xi n}{n^{\alpha}}\right|\geq
\left|\sum_{n=1}^{N_k}\frac{\sin^n \pi\xi_k n}{n^{\alpha}}-\sum_{n=1}^{N_k}\frac{\sin^n \pi\xi_k n-\sin^n \pi\xi n}{n^{\alpha}}\right|\geq
$$
$$
\geq\left|\sum_{n=1}^{N_k}\frac{\sin^n \pi\xi_k n}{n^{\alpha}}\right|-\sum_{n=1}^{N_k}\frac{|\sin^n \pi\xi_k n-\sin^n \pi\xi n|}{n^{\alpha}}.
$$
Applying Lemma \ref{divrate} with $q=2^{\nu_k}$, $L=2^{2q(A_q+k+4)}$, and $N=N_k=2qL$ yields an estimate
$$\left|\sum_{n=1}^{N_k}\frac{\sin^n \pi\xi_k n}{n^{\alpha}}\right|\geq\frac{1}{2q}\log_2 L-A_q=k+4$$
and the same one for the cos-series.
Next, by the mean value theorem
$$
\sum_{n=1}^{N_k}\frac{|\sin^n \pi\xi_k n-\sin^n \pi\xi n|}{n^{\alpha}}\leq\sum_{n=1}^{N_k}\frac{\pi n|\xi_k-\xi|}{n^{\alpha}}\leq
\pi |\xi - \xi_k|\sum_{n=1}^{N_k}n \leq \pi N_k^2|\xi - \xi_k|.
$$
Since by construction
$$
|\xi-\xi_k|\leq\frac{2}{2^{\nu_{k+1}}},
$$
in line with the choice of $\nu_{k+1}$ we have
$$
N_k^2|\xi - \xi_k|\leq
\left(2\cdot 2^{\nu_k}\cdot 2^{2\cdot2^{\nu_k}(A_{2^{\nu_k}}+k+4)}\right)^2\frac{2}{2^{2^{\nu_{k}+2}(A_{2^{\nu_{k}}}+k+4)+2\nu_{k}+3}}=1.
$$
Thus
$$
\left|\sum_{n=1}^{N_k}\frac{\sin^n \pi\xi n}{n^{\alpha}}\right|\geq k+4-\pi > k,
$$
which is the required result. The same lower bound for the partial sum of
the $\cos$-series can be easily obtained quite the same way.
\end{proof}

In conclusion we note that the set $E$ from Theorem \ref{divergence} consists solely of Liouville numbers.

{\it Department of Mathematics and Mechanics, Moscow State University

1 Leninskie Gory, 119991, Moscow, Russia}

{\tt ab@rector.msu.ru, goryashin@mech.math.msu.su}

\end{document}